\documentclass[11pt]{amsart}
\usepackage{amssymb,latexsym}
\usepackage{enumerate}
\usepackage[T1]{fontenc}
\usepackage{mathrsfs}
\usepackage{amsmath}
\usepackage{yhmath}
\usepackage{xcolor}
\makeatletter
\@namedef{subjclassname@2010}{%
  \textup{2010} Mathematics Subject Classification}
\makeatother
\makeindex

\usepackage{chngcntr}

\usepackage{mathtools}
\counterwithin{equation}{section}

\newtheorem{theorem}{Theorem}[subsection]

\theoremstyle{definition}
\newtheorem{df}[theorem]{Definition}

\newtheorem{rem}[theorem]{Remark}
\newtheorem{ex}[theorem]{Example}
\newtheorem{prob}[theorem]{Problem}

\numberwithin{theorem}{section}

\newcommand{\mb}{\mathbb}
\newcommand{\mc}{\mathcal}

\newcommand{\s}{\subset}
\newcommand{\bs}{\boldsymbol}

\begin{document}
\title{Families of exposing maps in strictly pseudoconvex domains}
\author{Arkadiusz Lewandowski}
\address{Institute of Mathematics\\ Faculty of Mathematics and Computer Science\\ Jagiellonian University\\ {\L}ojasiewicza 6,
30-348 Kraków, Poland}
\email{Arkadiusz.Lewandowski@im.uj.edu.pl}
\begin{abstract}
We prove that given a family $(G_t)$ of strictly pseudoconvex domains  varying in $\mc{C}^2$ topology on domains, there exists a continuously varying family of exposing maps $h_{t,\zeta}$ for all $G_t$ at every $\zeta\in\partial G_t.$  
\end{abstract}


\thanks{The author was supported by the grant UMO-2017/26/D/ST1/00126 financed by the National Science Centre, Poland}

\maketitle
\section{Introduction}

Let $G\s\s\mb{C}^n$ be a domain and let $\zeta\in\partial G.$ We say that $\zeta$ is a \emph{globally strongly convex} boundary point of $G$ if $\partial G$ is of class $\mc{C}^2$ and strongly convex at $\zeta$, and $\overline{G}\cap T_{\zeta}(\partial G)=\{\zeta\},$ where $T_{\zeta}(\partial G)$ denotes the tangent hyperplane of $\partial G$ at $\zeta.$ 
It is known (cf. [\ref{JGA14}]) that
\begin{theorem}\label{Exposing}
If $G$ is strictly pseudoconvex and has boundary of class $\mc{C}^2$, then for every $\zeta\in\partial G$ there exist a neighbourhood $\hat{G}$ of $\overline{G}$ and a holomorphic embedding $h:\hat{G}\rightarrow\mb{C}^n$ such that ${ h}(\zeta)$ is a globally strongly convex boundary point of ${ h}(G).$
\end{theorem}
\noindent Such an $h$ is called an \emph{exposing mapping of $G$ at $\zeta$.} The exposing maps are useful in the investigation of the boundary behaviour of the intrinsic metrics (see [\ref{FW}] or [\ref{Wo}]), in the studies on squeezing function (see, for example [\ref{DGZ}]), and in the proof of the boundary version of the open mapping theorem for holomorphic mappings between strictly pseudoconvex domains (see [\ref{BracFor}]). See also a survey article [\ref{Zhang}] and the references therein. \\
\indent A point $\zeta$ as above is called a \emph{peak point} with respect to $\mc{O}(\overline{G})$, the family of functions holomorphic in a neighborhood of $\overline{G},$ if there exists a function $f\in\mc{O}(\overline{G})$ such that $f(\zeta)=1$ and $f(\overline{G}\setminus\{\zeta\})\s\mb{D}:=\{z\in\mb{C}:|z|<1\}.$ Such an $f$ is called a \emph{peak function for $G$ at $\zeta$}.\\
\indent The following question has been formulated in [\ref{DGZ}]:
\begin{prob} Let $\rho:\mb{D}\times\mb{C}^n\rightarrow\mb{R}$ be a plurisubharmonic function of class $\mc{C}^{k}, k\in\mb{N},k\geq 2.$ Assume that for any $t\in\mb{D}$ the truncated function $\rho|_{\{t\}\times\mb{C}^n}$ is strictly plurisubharmonic and globally {defines a} bounded strictly pseudoconvex domain $G_t:=\{w\in\mb{C}^n:\rho(t,w)<0\}.$ This latter can be understood as a family of strictly pseudoconvex domains with boundaries of class $\mc{C}^k$ over $\mb{D}.$ 
Do there exist $\mc{C}^{k-2}$-continuously varying families: 
\begin{enumerate}[(A)]
\item $(f_{t,\zeta})_{t\in\mb{D},\zeta\in\partial G_t}$ of peak functions for $G_t$ at $\zeta\in\partial G_t$ 
\item $(h_{t,\zeta})_{t\in\mb{D},\zeta\in\partial G_t}$ of exposing maps for $G_t$ at $\zeta\in\partial G_t$?
\end{enumerate}
\label{Problem}
\end{prob}
In the papers [\ref{AL1}] and [\ref{AL2}] we have affirmatively answered the question (A). In [\ref{AL1}] we treated the particular case, where the parameter space $\mb{D}$ was replaced with some compact metric space, and the constructed family of peak functions was continuous with respect to the parameter (actually, it was continuous with respect to all variables). Later, in [\ref{AL2}], we considered the problem (A) in its full generality. The hereby paper is, in the author's intention, parallel 
to [\ref{AL1}] for the problem (B): we show that, under some additional assumption, given a family of domains $G_t$ as in Problem \ref{Problem}, for any compact $K\s\mb{D}$, there exists a continuous family $(h_{t,\zeta})_{t\in K,\zeta\in\partial G_t}$ of exposing maps for $G_t$ at $\zeta\in\partial G_t,t\in K.$ Namely, we prove    
\begin{theorem}\label{main}
Let $(G_t)_{t\in\mb{D}}$ be a family of strictly pseudoconvex domains as in Problem \ref{Problem}  with $k=2$. Let $\sigma\in(0,1)$. Take an $R>0$ such that $\bigcup_{t\in \sigma\overline{{\mb{D}}}}\overline{G}_t\s\s\mb{B}(0,R)$. Assume that there exist a $\mc{C}^2$-continuous  family $(\gamma_{t,\zeta})_{t\in\sigma\overline{\mb{D}},\zeta\in\partial G_t}$ of smooth embedded arcs $[0,1]\rightarrow \mb{C}^n$ such that $\gamma_{t,\zeta}(0)=\zeta,\gamma_{t,\zeta}(1)\in\mb{S}^{{2n-1}}(R)$ and $\gamma_{t,\zeta}(x)\in\mb{C}^n\setminus(\overline{G_t}\cup\mb{S}^{{2n-1}}(R)),x\in(0,1),$ for all $t\in\sigma\overline{\mb{D}}$ and $\zeta \in \partial G_t$. Then there exist a family $(h_{t,\zeta})_{t\in \sigma\overline{\mb{D}},\zeta\in \partial G_t}$ of exposing maps for $G_t$ at $\zeta$, continuous with respect to all variables.
\end{theorem}
Here and below $\mb{B}(a,R)$ stands for the open ball in $\mb{C}^n$ with center at $a$ and radius $R>0$, and $\mb{S}^{{2n-1}}(R):=\partial\mb{B}(0,R)$.
\begin{rem}
Our assumption concerning the $\mc{C}^2$-continuity of the family $(\gamma_{t,\zeta})_{t\in\sigma\overline{\mb{D}},\zeta\in\partial G_t}$ should be understood in the following way:\\
For each $t$ let $\Gamma_t$ be a neighbourhood of $\partial G_t$ with $\nabla r_t\neq 0$ on $\Gamma_t,$ where $r_t:=\rho(t,\cdot)$ and $\nabla r_t$ denotes its gradient. The neighbourhoods $\Gamma_t$ may be chosen to depend in a $\mc{C}^2$-continuous way on $t$.\\
Then there {exist positive constants $\sigma'\in(\sigma,1)$ and $\tilde{\varepsilon}$} such that the family $(\gamma_{t,\zeta})_{t\in\sigma\overline{\mb{D}},\zeta\in\partial G_t}$ may be extended to a $\mc{C}^2$-continuous family 
$$(\gamma_{t,\zeta})_{t\in\sigma' {\mb{D}},\zeta\in \bigcup_{|\kappa|{<}{\tilde{\varepsilon}}}\partial G^{(\kappa)}_t}$$ 
of smooth embedded arcs $[0,1]\rightarrow \mb{C}^n$ such that $\gamma_{t,\zeta}(0)=\zeta,\gamma_{t,\zeta}(1)\in\mb{S}^{{2n-1}}(R)$ and $\gamma_{t,\zeta}(x)\in\mb{C}^n\setminus(\overline{G^{(\kappa)}_t}\cup\mb{S}^{{2n-1}}(R)),x\in(0,1),$ for all $t\in\sigma' {\mb{D}}$ and $\zeta \in \partial G^{(\kappa)}_t,|\kappa|{<}{\tilde{\varepsilon}}$. Here, for small $|\kappa|$ we have put 
$$G^{(\kappa)}_t:=(G_t\setminus\Gamma_t)\cup\{z\in \Gamma_t:r_t(z)<\kappa\}.$$\label{C1assumption}
\end{rem}
\indent Notice that the assumption concerning the existence of the family $(\gamma_{t,\zeta})$ of suitable embedded arcs is completely in the spirit of Theorem 1.3 from [\ref{PAMS18}], which is a version of our result for a single domain. This kind of assumption is not present in Theorem \ref{Exposing}, which is indeed a "pointwise" result for single domain. It seems that the existing methods do not allow to relax this additional assumption, with the main obstruction being of rather topological nature. On the other hand, in certain subclasses of the class of strictly pseudoconvex domains the existence of such family of embedded arcs need not be assumed: in Example \ref{Example} we show that if the domains $G_t$ are all strongly linearly convex, then the family $(\gamma_{t,\zeta})$ can always be constructed and therefore the family $(h_{t,\zeta})_{t\in \sigma\overline{{\mb{D}}},\zeta\in \partial G_t}$ always exist. This latter should be compared with Theorem 1.4 from [\ref{JGA14}], which says that if a single domain $G\s\s\mb{C}^n$ is convex, smoothly bounded, and of finite type $2l (l\in\mb{N})$, then there exists a \emph{smooth} family $(h_{\zeta})_{\zeta\in \partial G}$ of exposing maps for $G$ at $\zeta$ (moreover, in such a case, each $h_\zeta$ may be chosen to be a holomorphic automorphism of $\mb{C}^n)$, and with Theorems 5.1 and 5.2 from [\ref{PAMS18}].\\ 
\indent We propose the proof of Theorem \ref{main}, which merges the methods of [\ref{JGA14}] with those from [\ref{PAMS18}]. In the first part of the proof, ideologically similar to Lemma 3.1 from [\ref{JGA14}], we deliver some parametric version of Narasimhan lemma (see [\ref{GK}] for another approach), thus constructing the family of local and locally exposing maps. Then, with the aid of [\ref{Ugolini}], we find the family of global and locally exposing maps. In the final part of the proof, based on ideas from [\ref{PAMS18}], we pass to the construction of the required family of exposing maps. The main tools will be the parametric version of the so-called Forstneri\v{c} splitting lemma for biholomorphic maps due to Simon (see Theorem \ref{Splitting}) and the following parametric version of higher-dimansional Mergelyan approximation theorem [\ref{Legacy}, Theorem 21]. Although some authors refer to certain parametric versions of Mergelyan theorem (cf. [\ref{JGA14}]), we were not able to find any in the literature. Also, versions referred in mentioned sources seem to be not suitable for our purposes.
\begin{theorem}\label{Mergelyan}
Let $S=K\cup M\s\mb{C}^n$ be admissible in the sense of \emph{[\ref{Legacy}]}, i.e. $S$ and $K$ {are Stein} compacts and $M$ is a totally real submanifold of class $\mc{C}^k$ (with boundary) with some $k\in\mb{N}$, let $(g_t)_{t\in T}\s\mc{C}^k(W)\cap\mc{O}(V)$ be a family of functions continuously dependent on all variables together with the parameter $t\in T$, where $T$ is a compact metric space and $V,W$ are some open neighbourhoods of $K,S$, respectively. Then there exists an open neighbourhood $\Omega$ of $S$ { such that for any $\varepsilon>0$ there exist }$(f_t)_{t\in T}\s\mc{O}(\Omega),$ a family of functions continuously dependent on all variables and such that for all $t\in T$ we have $\|g_t-f_t\|_{\mc{C}^k(S)}<\varepsilon.$ 
\end{theorem}
Theorem \ref{Mergelyan} is proved in Section \ref{MergelyanProof}, while the proof of our main result, Theorem \ref{main}, is presented in Section \ref{mainProof}. Finally, in Section \ref{SLCC} we discuss the case of the family of strongly linearly convex domains. We start with some preliminary results, presented in Section \ref{Prelim}, and end with some concluding remarks in Section \ref{concluding}.\\
\indent The author would like to thank his Teacher, Professor Marek Jarnicki for his encouragement to undertake this topic. He is also indebted to Andrzej Czarnecki and Andrea Spiro for valuable consultations concerning the issues related to Example \ref{Example}.
\section{Preliminaries}\label{Prelim}
\begin{df} Let $G\s\s\mb{C}^n$ be a domain. It is called a \emph{strictly pseudoconvex} if there exist a neighborhood $U$ of $\partial G$ and a \emph{defining function} $r:U\rightarrow\mb{R}$ of class $\mc{C}^2$ on $U$ and such that
\begin{enumerate}[(i)]
\item $G\cap U=\{z\in U:r(z)<0\}$,\label{Condition1}
\item $(\mb{C}^n\setminus\overline{G})\cap U=\{z\in U:r(z)>0\},\label{Condition2}$
\item $\nabla r(z)\neq 0$ for $z\in \partial D,$ where $\nabla r(z):=\left(\frac{\partial r}{\partial\overline{z}_1}(z),\cdots,\frac{\partial r}{\partial\overline{z}_n}(z)\right)$,\label{Condition3}
\end{enumerate}
together with $$\mc{L}_r(z;X)>0\text{\ for\ } z\in\partial G\text{\ and\ nonzero\ }X\in T_z^{\mb{C}}(\partial G),$$  
where $\mc{L}_r$ denotes the Levi form of $r$ and $T_z^{\mb{C}}(\partial G)$ is the complex tangent space to $\partial G$ at $z$.
\end{df}
\indent It is known that $U$ and $r$ can be chosen to satisfy (\ref{Condition1})-(\ref{Condition3}) and, additionally:
\begin{enumerate}[(iv)]
\item $\mc{L}_r(z;X)>0$ for $z\in U$ and all nonzero $X\in\mb{C}^n,$\label{Condition4}
\end{enumerate}
cf. [\ref{Kra1}]. \\
\indent Note that for a function $r$ as above and a point $\zeta\in\partial G$, the Taylor expansion of $r$ at $\zeta$ has the following form:
\begin{equation}
r(z)=r(\zeta)-2\text{Re}P_r(z;\zeta)+\mc{L}_{r}(\zeta;z-\zeta)+o(\|z-\zeta\|^2),\label{Taylor}
\end{equation}
where 
$$\displaystyle{
P_r(z;\zeta):=-\sum_{j=1}^n\frac{\partial r}{\partial z_j}(\zeta)(z_j-\zeta_j)-\frac{1}{2}\sum_{i,j=1}^n\frac{\partial^2r}{\partial z_i\partial z_j}(\zeta)(z_i-\zeta_i)(z_j-\zeta_j)
}$$
is the \emph{Levi polynomial of $r$ at $\zeta$}.\\
\indent In Section \ref{SLCC} we shall discuss the stronger notion than that of strictly pseudoconvex domains. Namely, we need the following
\begin{df}
A domain $G\s\s\mb{C}^n$ with $\mc{C}^2$ boundary is called \emph{strongly linearly convex} if there exists a defining function $r$ for $G$ with
$$
\mc{L}_r(z;X)>\left|\sum_{j,k=1}^n\frac{\partial^2 r}{\partial z_j\partial z_k}(z)X_jX_k\right|,
$$
for all $z\in\partial G$ and all nonzero $X=(X_1,\ldots, X_n)\in T_z^{\mb{C}}(\partial G).$
\end{df}
Finally, an important ingredient of the proof of Theorem \ref{main} is the parametric version of Forstneri\v{c}'s splitting lemma for biholomorphic maps close to identity. 
\begin{df}
A pair $(A,B)$ of compact subsets of $\mb{C}^n$ is called a \emph{Cartan pair}, if 
\begin{enumerate}[(i)]
\item $A,B,A\cup B,A\cap B$ are all Stein compacts,
\item $\overline{A\setminus B}\cap\overline{B\setminus A}=\varnothing.$
\end{enumerate}
\end{df}
The following  comes from [\ref{SimonPreprint}] (see also [\ref{SimonThesis}], and [\ref{Forstneric}] for a non-parameter version). Henceforth for a set $Z\s\mb{C}^n$ and a number $\delta>0$ we abbreviate $Z^{\delta}:=\bigcup_{z\in Z}\mb{B}(z,\delta)$.
\begin{theorem}
Let $\mc{T}$ be a nonempty compact topological space, let $(A_t,B_t)_{t\in\mc{T}}$ be an \emph{admissible} (in the sense of \emph{[\ref{SimonPreprint}]}) family of Cartan pairs, and let $\mu>0.$ Then there exists a $\tau>0$ such that for any $\eta>0$ there exists an $\varepsilon_{\eta}>0$ with the property that for any family $(\gamma_{t}:(A_t\cap B_t)^{\mu}\rightarrow\mb{C}^n)_{t\in\mc{T}}$ of injective holomorphic maps satisfying $\|\gamma_t-\text{\emph{\ Id}}\|_{(A_t\cap B_t)^{\mu}}<\varepsilon_{\eta},t\in\mc{T}$ and depending continuously with respect to all variables, there exist families $(\alpha_t:A_t^{2\tau}\rightarrow\mb{C}^n)_{t\in\mc{T}}, (\beta_t:B_t^{2\tau}\rightarrow\mb{C}^n)_{t\in\mc{T}}$ of injective holomorphic maps, depending continuously on all variables, and such that for all $t\in\mc{T}$ we have
\emph{\begin{enumerate}
\item \emph{$\gamma_t=\beta_t\circ\alpha_t^{-1}$ on $(A_t\cap B_t)^{\tau}$, and}
\item $\|\alpha_t-\text{\ Id}\|_{(A_t)^{2\tau}}<\eta, \|\beta_t-\text{\ Id}\|_{(B_t)^{2\tau}}<{\eta}$.
\end{enumerate}\label{Splitting}}
\end{theorem}
\begin{rem}
In the proof of Theorem \ref{main} we shall apply Theorem \ref{Splitting} in the situation where for sufficiently small positive $k$
$$
A_{t,\zeta}=\overline{D_{t,\zeta}}\cap\overline{\mb{B}}(\zeta,k),\quad B_{t,\zeta}=\overline{D_{t,\zeta}}\setminus{{\mb{B}}}(\zeta,\frac{k}{2}),
$$
with the indices $(t,\zeta)$ taken from a compact set of an Euclidean space and $(D_{t,\zeta})$ forming a family of strictly pseudoconvex domains varying in a $\mc{C}^2$-continuous manner, and with the property that $\zeta\in\partial D_{t,\zeta}.$ In this case, it is possible to choose $(\alpha_{t,\zeta})$ and $(\beta_{t,\zeta})$ as above, where the functions $\alpha_{t,\zeta}$ additionally interpolate identity at $\zeta$ to an arbitrarily high order, cf. remarks from Lemma 5.3 in [\ref{PAMS18}] and from Lemma 5.2 in [\ref{JGA14}]. \label{Interpolation}
\end{rem}
\section{Proof of Theorem \ref{Mergelyan}}\label{MergelyanProof}
The proof of Theorem \ref{Mergelyan} goes along the lines of the proof of Theorem 21 from [\ref{Legacy}], with the most important modification at the end, where we use methods of [\ref{Manne}], together with the parameter dependence of the solution of the 1st. Cousin Problem with bounds (to be deduced from the proof of Theorem VII.6.3 in [\ref{Ran}]). We {include} the proof here for the convenience of the {reader} and in order to be able to point out the modifications of it needed in getting the target described in Remark \ref{RemarkMergelyan} below. 
\begin{proof} 
\noindent\textit{Step 1.} Suppose that there exists a $t\in T$ such that spt$g_t\cap K\neq\varnothing$.\\
Take $\tilde{\Omega}$, a Stein neighbourhood of $S$ such that $K_0:=\hat{K}_{\mc{O}(\tilde{\Omega})}\s V$ (cf. Lemma 2 in [\ref{Legacy}]). Let $K_1\s V$ be a $\mc{O}(\tilde{\Omega})$-convex compact set such that $K_0\s$int$K_1$ and $s\bs{e}_{1}\notin K_1$ (this latter condition is not needed for the proof - we will use it later, cf. Remark \ref{RemarkMergelyan}). Choose $\chi$, a smooth cutoff function with support equal $K_1$ and with the property that $\chi=1$ in a neighbourhood of $K_0$. By Oka-Weil theorem with parameters (Theorem 2.8.1 in [\ref{Forstneric}]), there exists a continuous family $(\varphi_t)_{t\in T}\s\mc{O}(\tilde{\Omega})$ with
$$
\text{sup}_{z\in K_1,t\in T}|g_t(z)-\varphi_t(z)|<\varepsilon.
$$ 
Let 
$$
\tilde{g}_t:=\chi\varphi_t+(1-\chi)g_t=\varphi_t+(1-\chi)(g_t-\varphi_t).
$$
{Then $\|\tilde{g}_t-g_t\|_{\mc{C}^k(S)}< C\varepsilon$ for all $t\in T$ with positive constant $C$ depending only on $K$ and $\chi$.} Therefore, it remains to show that we are able to approximate the family $(\psi_t:=(1-\chi)(g_t-\varphi_t))_{t\in T}\s\mc{C}^k(S)$, which enjoys the property that the supports of the functions $\psi_t$ do not intersect some (common) neighbourhood of $K_0$. We therefore can pass to the following:\\
\noindent\textit{Step 2.} Approximation of $(\psi_t)_{t\in T}$ (if the condition in Step 1 is empty, we go directly to Step 2).\\
If for some $t\in T$ we have spt$\psi_t\cap\partial M\neq\varnothing$, we take a bigger totally real submanifold, still denoted by $M$, of class $\mc{C}^k$, by extending $M$ through $\partial M$, and we extend all functions $\psi_t{|_{M}}$ to the functions of class $\mc{C}^k$ with compact supports contained in {the} relative interior of ${M}$ (we keep the notation $\psi_t$ for these extensions). Note that this extension of $M$ may be taken one good for all $t\in T$ and the extensions of functions $\psi_t$ may be taken to depend continuously on all variables. Using now the fact that $\chi$ above equals 1 on a neighbourhood, say $V_0$, of $K_0$, we may multiply everything by {another cutoff function, obtaining} the existence of some compact set in $F\s M\setminus K$ such that for all $t\in T$ we have spt$\psi_t|_M\s F.$\\
We cover $F$ by a finite number of domains $M_1,\ldots, M_m\s M$ such that for every $j\in\{1,\ldots, m\}$ Proposition 2 from [\ref{Legacy}] (one may also bear in mind Proposition at the beginning of Section 3 from [\ref{Manne}]) holds true for $M_j$ (observe that the functions $f_{\varepsilon}$ constructed therein changes continuously if the input data are perturbed in a continuous way) and $\bigcup_{j=1}^m\overline{M}_j\s M\setminus K.$ Let $(\chi_j\in\mc{C}^k_0(M_j))_{j=1}^m$ be a partition of unity subordinated to the cover $M_1,\ldots,M_m$ of a neighbourhood of $F$ so that for all $t\in T$ we have $\psi_t=\sum_{j=1}^m\chi_j\psi_t.$ We see it suffices to approximate every single family $(\chi_j\psi_t)_{t\in T}$ with $j$ fixed. Without loss of generality, assume that spt$\psi_t\s\s M_1,t\in T$. Let $U\s\mb{C}^n$ be a neighbourhood of $\partial M_1$ as in (b) from Proposition 2 in [\ref{Legacy}] (observe it is independent of $\psi_t$). Take open sets $A,B\s\mb{C}^n$ such that
$$
M_1\s B,\quad S\setminus M_1\s A, \quad A\cap B\s U. 
$$   
Let $\Omega$ be a Stein neighbourhood of $S$ with $\Omega\s (A\cup B)\cap\tilde{\Omega}$ and define 
$$
\Omega_A:=\Omega\cap A,\quad \Omega_B:=\Omega\cap B.
$$
By shrinking $\Omega$ little bit, we may assume it is a strictly pseudoconvex domain with smooth boundary. Now, analyzing the proof of Theorem VII.6.3 from [\ref{Ran}], we see that the solution of the 1st. Cousin Problem with bounds depends continuously on parameters if only input functions are entire and taken to also depend continuously on parameters. Now, if we consider the family $(h_t)_{t\in T}\s\mc{O}(\mb{C}^n)$ of functions, continuously dependent on all variables, given by Proposition 2 from [\ref{Legacy}], with
$$
\|\psi_t-h_t\|_{\mc{C}^k(M{_1})}<\varepsilon\text{\ and\ }\|h_t\|_U<\varepsilon,\quad t\in T,
$$ 
we find, by above observation, the continuous (in all variables) families of functions $(h_{t}^A)_{t\in T}\in\mc{O}(\Omega_A)$ and $(h_{t}^B)_{t\in T}\in\mc{O}(\Omega_B)$ such that for a positive constant $D$, independent of $t\in T$, we have
$$
\|h^A_t\|_{\Omega_A}<D\varepsilon, \|h^B_t\|_{\Omega_B}<D\varepsilon,
$$
and
$$
h_t=h^A_t-h^B_t
$$
on $\Omega\cap A\cap B$ for all $t\in T.$ Finally we put $f_t:=h_t+h^B_t$ on $\Omega_B$ and $f_t:=h^A_t$ on $\Omega_A.$ Invoking the Cauchy estimates, we get the conclusion.
\end{proof}
\begin{rem}\label{RemarkMergelyan}
Actually, in the proof of Theorem \ref{main} we shall need stronger version of Theorem \ref{Mergelyan}, where the domains of definition of certain \emph{injections} $g_t$ will vary, the functions $g_t$ themselves will depend in a $\mc{C}^2$-continuous way on all variables and we will need to construct a family $f_t$ of holomorphic \emph{embeddings}, depending in a $\mc{C}^2$-continuous way on all variables and admitting an interpolation to order 3 at a certain point. We shall include the details within the proof of Theorem \ref{main}.  
\end{rem}
\section{Proof of Theorem \ref{main}}\label{mainProof}
\begin{proof}[Proof of Theorem \ref{main}] Let $r_t:=\rho(t,\cdot),t\in\mb{D}$. As in Remark \ref{C1assumption}, for each $t$ let $\Gamma_t$ be a neighbourhood of $\partial G_t$, $\mc{C}^2$-continuously dependent on $t$, and with $\nabla r_t\neq 0$ on $\Gamma_t.$ {Take positive $\sigma'\in(\sigma,1)$ and $\tilde{\varepsilon}$} such that the family $(\gamma_{t,\zeta})_{t\in\sigma\overline{\mb{D}},\zeta\in\partial G_t}$ may be extended to a $\mc{C}^2$-continuous family 
$$(\gamma_{t,\zeta})_{t\in\sigma' {\mb{D}},\zeta\in \bigcup_{|\kappa|{<\tilde{\varepsilon}}}\partial G^{(\kappa)}_t}$$ 
of smooth embedded arcs $[0,1]\rightarrow \mb{C}^n$ such that $\gamma_{t,\zeta}(0)=\zeta,\gamma_{t,\zeta}(1)\in\mb{S}^{{2n-1}}(R)$ and $\gamma_{t,\zeta}(x)\in\mb{C}^n\setminus(\overline{G^{(\kappa)}_t}\cup\mb{S}^{{2n-1}}(R)),x\in(0,1),$ for all $t\in {\mb{D}}$ and $\zeta \in \partial G^{(\kappa)}_t,|\kappa|{<\tilde{\varepsilon}}$.\\
\indent After possible decreasing of $\tilde{\varepsilon}$, there exists a $\mc{C}^2$-continuous family $$(l_{t,\zeta})_{t\in{\mb{D}},\zeta\in \bigcup_{|\kappa|{<\tilde{\varepsilon}}}\partial G^{(\kappa)}_t}$$ of global changes of coordinates, being compositions of translation and unitary transformation, such that $l_{t,\zeta}(\zeta)=0$ and with the property that $\bs{n}_{t,\zeta}$, a unit exterior normal vector to $\partial G_t^{(\kappa)}$ at $\zeta$ is transformed to a vector $(1,0\ldots,0).$ In particular, in these new coordinates we have $T^{\mb{C}}_{\zeta}(\partial G_t^{(\kappa)})=\{w_1=0\}.$ Write $l_{t,\zeta}(z)=\Phi_{t,\zeta}(z-\zeta)$ with $\Phi_{t,\zeta}$ being a unitary matrix depending $\mc{C}^2$-continuously on $t\in{\mb{D}}$ and $\zeta\in \partial G_t^{(\kappa)}.$\\
\indent Denote by $P_{t,\zeta}$ the Levi polynomial of $r_t$ at $\zeta.$ It is standard that there exist positive constants $C,\xi,$ and $\lambda$ such that for any $t\in\sigma'{\mb{D}}$, any $\zeta\in\Gamma_t$ with dist$(\zeta,\partial G_t)<\xi$, and any $z\in\mb{B}(\zeta,\lambda)$ we have
$$
r_t(z)\geq r_t(\zeta)-2{\text Re}P_{t,\zeta}(z)+C\|z-\zeta\|^2.
$$ 
In particular, if $\zeta\in\partial G_t$ and $z\in\overline G_t$ is such that $\|z-\zeta\|<{\lambda}$, we have
$$
-{\text Re}P_{t,\zeta}(z)\leq - \frac{C}{2}\|\zeta-z\|^2.
$$
Putting $\hat{P}_{t,\zeta}(z):=P_{t,\zeta}(l^{-1}_{t,\zeta}(z))$ we get the estimate
\begin{equation}
-\text{Re}\hat{P}_{t,\zeta}(z)\leq - \frac{C}{2}\|\zeta-l^{-1}_{t,\zeta}(z)\|^2=-\frac{C}{2}\|\Phi_{t,\zeta}^{-1}z\|^2=-\frac{C}{2}\|z\|^2\label{eq1}
\end{equation}
for $z\in l_{t,\zeta}(G_t)\cap\mb{B}(0,p)$ (and all $t\in\sigma'\mb{D}$ and $\zeta\in\partial G_t$) with some positive $p$. Consider the mapping 
$$\Psi_{t,\zeta}(z):=(-\hat{P}_{t,\zeta}(z),z_2,\ldots,z_n)=(-P_{t,\zeta}(l^{-1}_{t,\zeta}(z)),z_2,\ldots,z_n),$$
for $z=(z_1,\ldots,z_n)\in\mb{C}^n$. We have ${\Psi}_{t,\zeta}(0)=0$, ${\Psi}_{t,\zeta}$ is injective and holomorphic in $\mb{B}(0,\rho)$ with some positive $\rho$ (independent of $t\in\sigma'\mb{D},\zeta\in\partial G_t$). Also, { $\bs{n}_{t,\zeta}$ becomes $(1,0,\ldots,0)$ in the local coordinates near $\zeta$ given by ${\Psi}_{t,\zeta}\circ l_{t,\zeta}$.}\\
Define 
$$
\Omega_{t,\zeta}:=l_{t,\zeta}(G_t).
$$
Recall that the elements of $\Psi_{t,\zeta}(\Omega_{t,\zeta}\cap\mb{B}(0,\rho))$ are of the form $(-\hat{P}_{t,\zeta}(z),z_2,\ldots,z_n)$ for $z\in\Omega_{t,\zeta}\cap\mb{B}(0,\rho).$\\
{ For $z\in\Omega_{t,\zeta}\cap\mb{B}(0,\rho)$, writing $z_1=x_1+ix_2$ and $z'=(z_2,\ldots,z_n)$, we get from (\ref{eq1}) the following estimate
$$
-{\text Re}\hat{P}_{t,\zeta}(z)-\omega_{t,\zeta}(-\text{Im}\hat{P}_{t,\zeta}(z),z')\leq 0,
$$}
where 
$$
\omega_{t,\zeta}(x_2,z'):=-D(x_2^2+\|z'\|^2)
$$
with some positive constant $D$, independent of $t\in\sigma'\mb{D}$ and $\zeta\in\partial G_t$. In particular,
\begin{equation}
\overline{\Psi_{t,\zeta}(\Omega_{t,\zeta}\cap\mb{B}(0,\rho))}\s\{z\in\mb{C}^n:x_1-\omega_{t,\zeta}(x_2,z')\leq 0\},\label{eq2}
\end{equation}
which implies that near $0$ the domain $\Psi_{t,\zeta}(\Omega_{t,\zeta})$ is strictly 2-convex in the sense of [\ref{JGA14}].\\
\indent All the above constructions remain valid if we allow $\zeta$ not only from $\partial G_t$, but also from $\partial G_t^{(\kappa)}$, where $|\kappa|\leq\varepsilon$, with some positive constant 
${\varepsilon<\tilde{\varepsilon}}$ which is taken one good for all $t\in\sigma'\mb{D}$. Let us denote from now on
$$
N(\partial G_t):=\bigcup_{|\kappa|<\frac{\varepsilon}{2}}\partial G_t^{(\kappa)},\quad t\in\sigma'\mb{D}.
$$
Observe that for fixed $t$ and $\zeta\in N(\partial G_t)$ there exists only one $\kappa$ with $\zeta\in\partial G_t^{(\kappa)}.$ Therefore, fixing a pair $(t,\zeta)$ actually determines the triple $(t,\zeta, \kappa).$ We shall frequently use this fact in the sequel without additional comments: namely, we shall only sometimes - when it will be not clear from the context - write $\kappa(t,\zeta)$ to indicate the fact that the particular $\kappa$ arises from the choice of $(t,\zeta)$. Otherwise, we shall omit this indexing.\\ 
Analyzing the proof of Proposition 1.2 in [\ref{Ugolini}], we see that there exists a $\mc{C}^2$-continuous family $(\Theta_{t,\zeta})_{t\in\sigma'\mb{D},\zeta\in N(\partial G_t)}$ of holomorphic automorphisms of $\mb{C}^n$, that can be represented as
$$
\Theta_{t,\zeta}(z)=\Psi_{t,\zeta}(z)+f_{t,\zeta}(z),
$$ 
where $f_{t,\zeta}$ is entire, and there exist positive constants $\theta$ and $A$, independent of $t\in\sigma'\mb{D}$ and $\zeta\in N(\partial G_t)$ with the property that for $\|z\|\leq \theta$ we have $\|f_{t,\zeta}(z)\|\leq A\|z\|^3$ (note {that} the set of parameters in [\ref{Ugolini}] is assumed to be a Stein space. On the other hand, the dependence of parameters considered there is holomorphic. We need only a $\mc{C}^2$-continuous dependence of parameters and in this case the assumption about {Steinness} of the set of parameters can be omitted. Also, the very same conclusion is possible to be obtained by using [\ref{Weickert}]).\\
\indent For $t\in\sigma'\mb{D}$ and $\zeta\in\partial G_t^{(\kappa)}$ put $D_{t,\zeta}^{(\kappa)}:=\Theta_{t,\zeta}(l_{t,\zeta}(G_t^{(\kappa)}))$ and observe that for these domains we have $T_0(\partial D_{t,\zeta}^{(\kappa)})$, the real tangent space to $\partial D_{t,\zeta}^{(\kappa)}$ at $0$, equals $\{x_1=0\}$, $\bs{n}_0=(1,0,\ldots,0),$ and $\partial D_{t,\zeta}^{(\kappa)}$ is 2-convex near $0$ in the sense of (\ref{eq2}), {after possible decreasing $\rho$ and $D$} there, so that they remain to be independent of $t\in\sigma'\mb{D}, |\kappa|{<}\varepsilon${,} after possible decreasing $\varepsilon$, and $\zeta\in\partial G_t^{(\kappa)}$.\\
\indent As in [\ref{PAMS18}], we may without loss of generality (after possible slight decreasing of $\sigma'$) modify the family of curves $(\gamma_{t,\zeta})_{t\in\sigma'{\mb{D}},\zeta\in N(\partial G_t)}$  so that the initial parts (of uniform length) of the curves $\Theta_{t,\zeta}(l_{t,\zeta}(\gamma_{t,\zeta}))$ are all equal to the segment $\bs{e}_{1}[0,s]$ with some positive $s$, arbitrarily small{, and so that they are perpendicular to $\Theta_{t,\zeta}(l_{t,\zeta}(\mb{S}^{{2n-1}}(R)))$ where they intersect}. 
The modification may be carried out so that the modified family, still denoted by $(\gamma_{t,\zeta})_{t\in\sigma'\mb{D},\zeta\in N(\partial G_t)}$, keeps its regularity.\\
\indent For any $t\in\sigma'\mb{D}$ and any $\zeta\in N(\partial G_t)$ (after possible decreasing $\varepsilon$) let $U_{t,\zeta}$ be a neighbourhood of $\overline{D^{(\kappa)}_{t,\zeta}}\cup(\bs{e}_1[0,s])$ and $V_{t,\zeta}\s U_{t,\zeta}$ be a neighbourhood of $\overline{D^{(\kappa)}_{t,\zeta}}$, both $\mc{C}^2$-continuously dependent on $(t,\zeta)$, and let
a $\mc{C}^2$-continuous family $(g_{t,\zeta}:U_{t,\zeta}\rightarrow\mb{C}^n)_{t\in\sigma'\mb{D},\zeta\in N(\partial G_t)}$ of smooth embeddings such that $g_{t,\zeta}=$ Id in $V_{t,\zeta}$, $g_{t,\zeta}$ stretches $\bs{e}_{1}[0,s]$ to cover $\Theta_{t,\zeta}(l_{t,\zeta}(\gamma_{t,\zeta}{([0,1])}))$, and {$g_{t,\zeta}(\bs{e}_{1}[0,s])$ is perpendicular to $(\Theta_{t,\zeta}\circ l_{t,\zeta})(\mb{S}^{{2n-1}}(R))$}. Note that $U_{t,\zeta}$ and $V_{t,\zeta}$ may be chosen to be independent {of} $|\kappa|<\varepsilon${, and thus of $\zeta$} (after eventually decreasing $\varepsilon$).\\ 
\indent We want to apply Theorem \ref{Mergelyan} for small $\varepsilon_0>0$ and the family $(g_{t,\zeta})_{(t,\zeta)\in P},$ where $P$ is some relatively compact subset of the open set $\{(t,\zeta):t\in\sigma'\mb{D},\zeta\in N(\partial G_t)\}$ containing in its interior all the couples $(t,\zeta)$ with $t\in\sigma\overline{\mb{D}},\zeta\in\partial G_t$. 
In our concrete situation we want to modify the construction of $(f_{t,\zeta})$ therein, taking into account the variable domains of the functions $g_{t,\zeta}$ and in order to get its $\mc{C}^2$-continuous dependence on all variables and thus allowing the interpolation to order 3 at $s\bs{e}_{1}$ with the continuity of new approximating and interpolating family of functions with respect to all variables. Of course the domains of such functions will also depend on $t$ in a suitable way. Moreover, we need to make sure that the functions $f_{t,\zeta}$ are injections in some suitable chosen neighbourhoods of $\overline{D^{(\kappa)}_{t,\zeta}}\cup(\bs{e}_{1}[0,s])$ for $(t,\zeta)\in P$. Below we indicate the modifications of the construction of the family $(f_{t,\zeta})$ required when proving a variant of Theorem \ref{Mergelyan} in this particular case.\bigskip\bigskip\\
\noindent\textbf{Variable domains, the $\mc{C}^2$ dependence on all variables and the interpolation at $s\bs{e}_{1}.$} Fix $P$, a relatively compact subset of the open set $\{(t,\zeta):t\in\sigma'\mb{D},\zeta\in N(\partial G_t)\}$ containing in its interior all the couples $(t,\zeta)$ with $t\in\sigma\overline{\mb{D}},\zeta\in\partial G_t$. There exists a small positive constant $\beta<\frac{s}{3}$ with the property that for all couples $(t,\zeta)\in P$ we may find the strictly pseudoconvex domains $H_{t,\zeta}$ with $\mc{C}^2$-smooth boundaries, depending in a $\mc{C}^2$-continuous way on $(t,\zeta)$, and satisfying
\begin{enumerate}
\item $\overline{D^{(\kappa)}_{t,\zeta}}\s\s H_{t,\zeta}\s\s V_{t,\zeta}$
\item dist$(\overline{D^{(\kappa)}_{t,\zeta}},\partial H_{t,\zeta})=\beta$
\item dist$(\overline{H_{t,\zeta}},\partial V_{t,\zeta})\geq\beta$
\item tube with radius $\frac{\beta}{4}$ around $\bs{e}_1[0,s]$ is compactly contained in $U_{t,\zeta}$
\item {$\overline{H_{t,\zeta}}\cap(\bs{e}_1(\beta,s])=\varnothing$ and} $S_{t,\zeta}:=\overline{H_{t,\zeta}}\cup(\bs{e}_1[\beta,s])$ satisfies the assumptions of Theorem \ref{Mergelyan}.
\end{enumerate}
In order to apply the standard interpolation corrections with the continuity of corrected family of functions with respect to all variables, we need to adjust the construction of $(f_{t,\zeta})$ in Theorem \ref{Mergelyan} (for $S=S_{t,\zeta}$) so that we take care of the variable domains $U_{t,\zeta}$ and that it will depend in a $\mc{C}^2$-continuous way on all variables (as the input data $(g_{t,\zeta})$ do). To get this aim, observe that in our particular situation, in Step 1 of proof of Theorem \ref{Mergelyan} we may take $\varphi_{t,\zeta}=$ Id, $(t,\zeta)\in P$ (and a suitable family $(\chi_{t,\zeta})_{(t,\zeta)\in P}$ of cutoff functions, smoothly dependent on all variables, with supports contained in $V_{t,\zeta}$ and equal one on neighbourhoods of $\overline{H_{t,\zeta}}$ with - by the compactness argument - distances to the boundaries uniformly bounded from below). Putting now $\psi_{t,\zeta}:=(1-\chi_{t,\zeta})(g_{t,\zeta}-$Id$)_{(t,\zeta)\in P}$, we observe that there exist an $\tilde{s}>s$ and a compact set $F\s\bs{e}_1(\beta,\tilde{s}]$ such that for all $(t,\zeta)\in P$ we have spt$\psi_{t,\zeta}|_{\bs{e}_1[\beta,\tilde{s}]}\s F$ (after possible multiplying by suitable cutoff function) and it suffices to approximate the family $(\psi_{t,\zeta})_{(t,\zeta)\in P}.$ Therefore, we only have to modify the construction from Step 2 in order to get the better regularity we are after. Here, observe that the proof remains unchanged until we have to choose the open sets $M_1\s B$ and $S_{t,\zeta}\setminus M_1\s A$ with $A\cap B\s U$ for all $(t,\zeta)\in P$. Observe that, by the compactness argument, these sets may be chosen to be independent of $(t,\zeta)\in P$.\\
For a fixed $(t,\zeta)\in P$ one may choose a strictly pseudoconvex and smoothly bounded domains $N_{t,\zeta}$ and $\hat{N}_{t,\zeta}$ such that
$$S_{t,\zeta}\s\s N_{t,\zeta}\s\s\hat{N}_{t,\zeta}\s A\cup B,$$
and by Theorem V.2.5 from [\ref{Ran}], there exist neighbourhoods
$$
\overline{N_{t,\zeta}}\s V^{t,\zeta}_0\s\s V^{t,\zeta}\s\s\hat{N}_{t,\zeta}
$$
and a solution operator for $\overline{\partial}$-problem 
$$
\boldsymbol{T}_1^{V^{t,\zeta},V^{t,\zeta}_0}:\mc{C}_{0,1}(\overline{V^{t,\zeta}})\rightarrow\mc{C}_{0,0}(V^{t,\zeta}_0)
$$
satisfying (i)-(iii) therein. Observe that for $(s,\xi)$ sufficiently close to $(t,\zeta)$
we have 
$$
S_{s,\xi}\s\s N_{t,\zeta},
$$
with the distance to the boundary uniformly bounded from below.\\
\indent Now the family $(h_{s,\xi})\s\mc{O}(\mb{C}^n)$ from Step 2 of the proof of Theorem \ref{Mergelyan}, appearing there as $(h_t)$, depends in our situation in a $\mc{C}^2$-continuous way on all variables and it suffices to get the same regularity with respect to the parameters of the solutions of the 1st. Cousin Problem with bounds for functions $h_{s,\xi}$ and the coverings $N_{{t,\zeta},A}:=N_{{t,\zeta}}\cap A,N_{{t,\zeta},B}:=N_{{t,\zeta}}\cap B$ of $N_{{t,\zeta}}$. This we reach by observing that utilizing in the proof of Theorem VII.6.3 from [\ref{Ran}] the operator $\boldsymbol{T}_1^{V^{t,\zeta},V_{0}^{t,\zeta}}$ instead of $\hat{\boldsymbol{S}}_1$ (cf. [\ref{Ran}, Theorem VII.5.6]) gives, for $(s,\xi)$ close to $(t,\zeta)$, the functions $h^{{t,\zeta},A}_{s,\xi}\in\mc{O}(N_{{t,\zeta},A}), h^{{t,\zeta},B}_{s,\xi}\in\mc{O}(N_{{t,\zeta},B}),$ depending in a $\mc{C}^2$-continuous way on all variables, and such that $h_{s,\xi}=h_{s,\xi}^{{t,\zeta},A}-h_{s,\xi}^{{t,\zeta},B}$ on $N_{t,\zeta}\cap A\cap B$ and with
\begin{equation}
\|h^{{t,\zeta},A}_{s,\xi}\|_{N_{{t,\zeta},A}}<E\varepsilon,\|h^{{t,\zeta},B}_{s,\xi}\|_{N_{{t,\zeta},B}}<E\varepsilon,\label{eq3}
\end{equation} 
where the positive constant $E$ is independent of $s$ and $\xi$. Note that we have used (iii) from Theorem V.2.5 in [\ref{Ran}] (to get suitable regularity with respect to the parameters) and estimates from the beginning of Section V.3.2, also in [\ref{Ran}] (to get estimates (\ref{eq3})). Observe it is crucial that the functions $h_{s,\xi}$ are entire.\\
Then, by the compactness argument, we find $W_1,\ldots, W_q$, and open cover of a neighbourhood of $P$ such that for each $j=1,\ldots, q$ there exist a strictly pseudoconvex and smoothly bounded domain $N_j$ with the property that for all $(s,\xi)\in W_j$ {we have} $S_{s,\xi}\s {N}_j$ with the distance to the boundary uniformly bounded from below, {and} there exist functions $h^{j,A}_{s,\xi}\in\mc{O}(N_{j}\cap A), h^{j,B}_{s,\xi}\in\mc{O}(N_{j}\cap B),$ depending in a $\mc{C}^2$-continuous way on all variables, and such that $h_{s,\xi}=h_{s,\xi}^{j,A}-h_{s,\xi}^{j,B}$ on $N_{j}\cap A\cap B$ and with
$$
\|h^{{j},A}_{s,\xi}\|_{N_{j}\cap A}<E_j\varepsilon,\|h^{j,B}_{s,\xi}\|_{N_{j}\cap B}<E_j\varepsilon,
$$
with positive constant $E_j$ good for all $(s,\xi)\in W_j$.\\
Let $(p_j)_{j=1}^q$ be a partition of unity subordinated to the covering $(W_j)_{j=1}^q$ of a neighbourhood of $P$. Define for $(t,\zeta)\in P$
$$
h^A_{t,\zeta}:=\sum_{j=1}^qp_j(t,\zeta)h^{j,A}_{t,\zeta},\quad h^B_{t,\zeta}:=\sum_{j=1}^qp_j(t,\zeta)h^{j,B}_{t,\zeta}.
$$
Then $h^A_{t,\zeta}\in\mc{O}(Z_{t,\zeta}\cap A),h^B_{t,\zeta}\in\mc{O}(Z_{t,\zeta}\cap B)$, where $Z_{t,\zeta}$ is a neighbourhood of $S_{t,\zeta}$, $\mc{C}^2$-continuously dependent on $(t,\zeta)$ and with the distance to the boundary uniformly (in $(t,\zeta)$) bounded from below. Moreover,
$$
\|h^{A}_{t,\zeta}\|_{U_{t,\zeta}\cap A}<E'\varepsilon,\|h^{B}_{t,\zeta}\|_{U_{t,\zeta}\cap B}<E'\varepsilon,
$$
where positive constant $E'$ does not depend on $(t,\zeta).$ Also,
$$
h_{t,\zeta}=h^A_{t,\zeta}-h^B_{t,\zeta}
$$
on $Z_{t,\zeta}\cap A\cap B$ and we finish the proof along the lines of the proof of Theorem \ref{Mergelyan}.\\
Finally we may add a family of small corrections to get interpolation at $s\bs{e}_{1}$, which now depends continuously on all variables.\bigskip\bigskip\\

\noindent{\textbf{Injectivity in neighbourhoods of $\overline{D_{t,\zeta}^{(\kappa)}}\cup(\bs{e}_{1}[0,s])$ with uniform distance to the boundary.}} This is a consequence of suitably modified techniques presented in Lemma 2.3 from [\ref{PAMS18}]. Namely, let us fix $(t,\zeta)\in P$ and consider the restriction of the function $g_{t,\zeta}$ to the domain $U_{t,\zeta}^0$ with
$$
\overline{D_{t,\zeta}^{(\kappa)}}\s\s U_{t,\zeta}^0\s\s U_{t,\zeta},
$$
created by attaching to $H_{t,\zeta}$ a tubular neighbourhood of radius $\frac{\beta}{4}$ around $\bs{e}_{1}[0,s]$. Then for $(s,\xi)$ sufficiently close to $(t,\zeta)$
we have
$$
\overline{D_{s,\xi}^{(\kappa(s,\xi))}}\s\s U_{t,\zeta}^0
$$
as well as
$$
\overline{D_{t,\zeta}^{(\kappa)}}\s\s U_{s,\xi}^0.
$$
{We claim that for every $(s,\xi)$ sufficiently close to $(t,\zeta)$ there exists $W_{s,\xi}$, a neighbourhood of $\overline{D^{(\kappa(s,\xi))}_{s,\xi}}\cup(\bs{e}_{1}[0,s])$, such that the distance to the boundary is uniformly bounded from below and with the property that the functions $f_{s,\xi}$ given by Theorem \ref{Mergelyan} for ${g}_{s,\xi}$ with $\varepsilon<\varepsilon_0$ for some sufficiently small $\varepsilon_0>0$ and with the modifications described in the preceding paragraph are all injections on the domains $W_{s,\xi}$.} 
Indeed, let $\varepsilon_0>0$ and for $(s,\xi)$ close to $(t,\zeta)$ take the decreasing sequences of domains $(U_{(s,\xi),k})_{k\in\mb{N}}$ created in a similar way as $U^0_{s,\xi}$, only with the radius of the used tube less than $\frac{1}{k}$ and with $H_{s,\xi}$ replaced by $\overline{D^{(\kappa(s,\xi))}_{s,\xi}}^{\delta+\frac{1}{k}}$ with sufficiently small positive $\delta$ (so that it is compactly contained in $H_{s,\xi}$ for large $k$).\\ 
\indent {Suppose that for any $k\in\mb{N}$ and any $r\in\mb{N}$ large enough to ensure $U_{(s,\xi),k}\s {Z}_{s,\xi}\cap {Z}_{t,\zeta}$ {for $(s,\xi)$ with the distance to $(t,\zeta)$ smaller than $\frac{1}{r}$} (recall ${Z}_{s,\xi}$ are {in place of $\Omega$} in Theorem \ref{Mergelyan} - see also preceding paragraph - and they do not depend on $\varepsilon$), and any $j_0\in\mb{N}$ there are $\mb{N}\ni j\geq j_0$ and} $(t_{k,j,r},\zeta_{k,j,r})\in P$ with the distance to $(t,\zeta)$ smaller than $\frac{1}{r}$, and $a_{k,j,r}\neq b_{k,j,r}\in U_{(t_{k,j,r},\zeta_{k,j,r}),k}$ such that $f^j_{t_{k,j,r},\zeta_{k,j,r}}(a_{k,j,r})=f^j_{t_{k,j,r},\zeta_{k,j,r}}(b_{k,j,r})$, where the functions $f_{s,\xi}^j$ are given by Theorem \ref{Mergelyan} with $\varepsilon=\frac{1}{j}$ (after modifications described in the preceding paragraph concerning the regularity with respect to parameters and the interpolation condition). 
We may assume that $(t_{k,j,r},\zeta_{k,j,r})\to (t,\zeta)$ and $a_{k,j,r}\to a,b_{k,j,r}\to b$, where $a,b\in\overline{D^{(\kappa{(t,\zeta)})}_{t,\zeta}}^{\delta}\cup(\bs{e}_{1}[0,s])$ as $k,j_{{0}},r\to\infty.$\\ 
\indent Using the injectivity of ${g}_{t,\zeta}$ and the fact that, by the construction, the family $(f^j_{s,\xi})$ is uniformly bounded (and hence equicontinuous) near $a$ and near $b$ for $(s,\xi)$ sufficiently close to $(t,\zeta)$, we conclude that in fact it has to be $a=b.$ From {the mathods used in the proof of} Lemma 2.3 in [\ref{PAMS18}] we get $a$ may only be an element of the segment $\bs{e}_{1}(\delta,s].$ Now, using the properties of the domains $U_{(s,\xi),k}$, the equicontinouity of the family $(f^j_{t,\zeta})_{(t,\zeta)\in P, j\in\mb{N}}$ near $a$, and performing the computations similar to those from the proof of Lemma 2.3 in [\ref{PAMS18}], we get $f^j_{t_{k,j,r},\zeta_{k,j,r}}(a_{k,j,r})\neq f^j_{t_{k,j,r},\zeta_{k,j,r}}(b_{k,j,r})$ for large $k,j,r$, a contradition.\\ 
\indent Therefore, {for $(s,\xi)$ close to $(t,\zeta)$ there exist neighbourhoods $W_{s,\xi}$ of $\overline{D^{(\kappa(s,\xi))}_{s,\xi}}\cup(\bs{e}_{1}[0,s])$, such that the distance to the boundary is uniformly bounded from below and with the property that the functions $f_{s,\xi}$ given by Theorem \ref{Mergelyan} for ${g}_{s,\xi}$ with $\varepsilon<\varepsilon_0$ with sufficiently small $\varepsilon_0>0$, and with the modifications described in the preceding paragraph are all injections on the domains $W_{s,\xi}$.} Using the compactness argument, we see that for arbitrarily small $\varepsilon$ and for all $(t,\zeta)\in P$ there exist neighbourhoods $W_{t,\zeta}$ of $\overline{D^{(\kappa)}_{t,\zeta}}\cup(\bs{e}_{1}[0,s])$, such that the distance to the boundary is uniformly bounded from below and with the property that the functions $f_{t,\zeta}$ given by Theorem \ref{Mergelyan} for ${g}_{t,\zeta}$ with $\varepsilon$ are injections on the domains $W_{t,\zeta}.$ Moreover, for $\varepsilon$ sufficiently small, the domains $W_{t,\zeta}$ do not depend on $\varepsilon$.\bigskip\bigskip\\
\indent To summarize: we proved the existence of a continuous family $(f_{t,\zeta})_{(t,\zeta)\in P}$ of holomorphic embeddings $\overline{D^{(\kappa)}_{t,\zeta}}\cup(\bs{e}_{1}[0,s])\rightarrow\mb{C}^n$ uniformly (in all variables) close to Id on neighbourhoods of  $\overline{D^{(\kappa)}_{t,\zeta}}$, with the images $f_{t,\zeta}(\bs{e}_{1}[0,s])$ uniformly (in all variables) close to $\Theta_{t,\zeta}(l_{t,\zeta}(\gamma_{t,\zeta}{[0,1]}))$ and with the property that  {$f_{t,\zeta}(\bs{e}_{1}[0,s])$ is perpendicular to $(\Theta_{t,\zeta}\circ l_{t,\zeta})(\mb{S}^{{2n-1}}(R))$}. Note that in the case {of} a single domain treated in [\ref{PAMS18}], the similar result is obtained by different method, based on [\ref{KutWo}].\bigskip\bigskip\\
Let 
$$
\hat{h}_{t,\zeta}:=(\Theta_{t,\zeta}\circ l_{t,\zeta})^{-1}\circ f_{t,\zeta}\circ(\Theta_{t,\zeta}\circ l_{t,\zeta}),
$$
on the set where the composition is defined. In the last part of the proof we shall construct, for sufficiently small positive $s$, the continuous family $(F_{t,\zeta})_{(t,\zeta)\in P}$ of holomorphic embeddings $\overline{G_t^{{(\kappa)}}}\rightarrow\mb{C}^n$ such that for all $(t,\zeta)\in P$ we have
\begin{enumerate}[(1')]
\item $F_{t,\zeta}(\zeta)=(\Theta_{t,\zeta}\circ l_{t,\zeta})^{-1}(s\bs{e}_1)$\label{1prim}
\item $F_{t,\zeta}(\overline{G_t^{{(\kappa)}}}\cap\mb{B}(\zeta,s'))\s (\Theta_{t,\zeta}\circ l_{t,\zeta})^{-1}((\bs{e}_1[0,s-2b])^c\cup\mb{B}(\bs{e}_1(s-b),b)\cup\{s\bs{e}_1\})$
\item $\|F_{t,\zeta}-\text{Id}\|_{\overline{G_t^{{(\kappa)}}}\setminus\mb{B}(\zeta,s')}$ is arbitrarily small\label{3prim}
\end{enumerate}
with some small positive $s',b,c$. Taking, for all $t\in\sigma\overline{\mb{D}}$ and $\zeta\in\partial G_t$, the composition
$$
h_{t,\zeta}:=\hat{h}_{t,\zeta}\circ F_{t,\zeta}
$$
will end the proof.\bigskip\bigskip\\ 
\noindent{\textbf{Construction of the family $(F_{t,\zeta})_{(t,\zeta)\in P}$.}} 
Observe that there exist $r,S>0$ such that for all $(t,\zeta)\in P$ we have
$$
\Theta_{t,\zeta}(l_{t,\zeta}((\overline{G_t^{{(\kappa)}}}\cap\mb{B}(\zeta,r))\setminus\{\zeta\}))\s\mb{B}(-S\bs{e}_1,S)
$$
and
$$
\Theta_{t,\zeta}(l_{t,\zeta}(\zeta))=0\in\partial\mb{B}(-S\bs{e}_1,S).
$$
By Theorem 3.1 from [\ref{PAMS18}], for all sufficiently small positive $s,\delta,c$, and $b$ with $b<s,c<s$ there exists a holomorphic injection $\Phi:\overline{\mb{B}}(-S\bs{e}_1,S)\rightarrow\mb{C}^n$ such that 
\begin{enumerate}[(1+)]
\item $\|\Phi-\text{Id}\|_{\overline{\mb{B}}(-S\bs{e}_1,S)\setminus\mb{B}(0,\delta)}$ is arbitrarily small
\item $\Phi(0)=s\bs{e}_1$
\item $\Phi(\overline{\mb{B}}(-S\bs{e}_1,S)\cap\mb{B}(0,\delta))\s(\bs{e}_1[0,s-2b])^c\cup\mb{B}((s-b)\bs{e}_1,b)\cup\{s\bs{e}_1\}$.
\end{enumerate}
Define, for $(t,\zeta)\in P$,
$$
\tilde{F}_{t,\zeta}(z):=((\Theta_{t,\zeta}\circ l_{t,\zeta})^{-1}\circ\Phi\circ(\Theta_{t,\zeta}\circ l_{t,\zeta}))(z)
$$
for $z$ from $\mb{B}(\zeta,r)$ intersected with some neighbourhood of $\overline{G_t^{{(\kappa)}}}$ of uniform size (in $t,\zeta$). Then, after eventually shrinking the domains of definition, the family $(\tilde{F}_{t,\zeta})_{(t,\zeta)\in P}$ depends $\mc{C}^2$-continuously on all variables and for every $(t,\zeta)\in P$ we have
\begin{enumerate}[(1*)]
\item $\tilde{F}_{t,\zeta}$ is a holomorphic injection
\item $\tilde{F}_{t,\zeta}(\zeta)=(\Theta_{t,\zeta}\circ l_{t,\zeta})^{-1}(s\bs{e}_1)$
\item $\tilde{F}_{t,\zeta}(\overline{G_t^{{(\kappa)}}}\cap\mb{B}(\zeta,s'))\s(\Theta_{t,\zeta}\circ l_{t,\zeta})^{-1}((\bs{e}_1[0,s-2b])^c\cup\mb{B}((s-b)\bs{e}_1,b)\cup\{s\bs{e}_1\})$ with some sufficiently small $s'\in(0,s).$\label{3star}
\end{enumerate}
Let us consider, for sufficiently small positive $k$, the family of Cartan pairs $((A_{t,\zeta},B_{t,\zeta}))_{(t,\zeta)\in P}$, where
$$
A_{t,\zeta}:=\overline{G_t^{(\kappa)}}\cap\overline{\mb{B}}(\zeta,k),\quad B_{t,\zeta}:=\overline{G_t^{(\kappa)}}\setminus{{\mb{B}}}(\zeta,\frac{k}{2}).
$$
Define
$$
C_{t,\zeta}:=A_{t,\zeta}\cap B_{t,\zeta}
$$
and observe that if $k$ is small enough, then we have
$$
\Theta_{t,\zeta}(l_{t,\zeta}(C_{t,\zeta}))\s\s\mb{B}(-S\bs{e}_1,S)
$$
with the distance to the boundary uniformly (in $t,\zeta$) bounded from below. If in the choice of $\Phi$ the constant $\delta$ (and consequently also $s'$ from (\ref{3star}*)) was sufficiently small, then for all $(t,\zeta)\in P$ the mappings $\tilde{F}_{t,\zeta}$ are uniformly arbitrarily close to Id in neighbourhoods of $C_{t,\zeta}$ of uniform (in $t,\zeta$) size. By Theorem \ref{Splitting} and Remark \ref{Interpolation} we get the existence of the continuous family $(\alpha_{t,\zeta})_{(t,\zeta)\in P}$ of biholomorphic mappings in neighbourhoods of $A_{t,\zeta}$, of uniform (in $t,\zeta$) size, interpolating {the} identity to {arbitrarily high order} at $\zeta$, and the continuous family $(\beta_{t,\zeta})_{(t,\zeta)\in P}$ of biholomorphic mappings in neighbourhoods of $B_{t,\zeta}$ of uniform (in $t,\zeta$) size, such that the family
$$
F_{t,\zeta}:=
\begin{cases}
\tilde{F}_{t,\zeta}\circ\alpha_{t,\zeta},& \text{in\ a\ neighbourhood\ of\ }A_{t,\zeta}\\
\beta_{t,\zeta},& \text{in\ a\ neighbourhood\ of\ }B_{t,\zeta}
\end{cases},\quad (t,\zeta)\in P,
$$
fulfilling (\ref{1prim}')-(\ref{3prim}'), is the last piece of our puzzle.
\end{proof}
\section{Strongly linearly convex case}\label{SLCC}
\begin{ex}\label{Example}
Let $k\geq 3$ and let $\rho$ and $G_t$ be as in Problem \ref{Problem} and assume additionally that $G_t$ is strongly linearly convex for each $t$. Let $\sigma<\sigma'\in(0,1).$\\ 
By Proposition 2.2.3 in [\ref{Extension}], for any $t$ there exist $U_t$, a neighbourhood of $\partial G_t$ and a $\mc{C}^{k-1}$-continuous mapping $\pi_t:U_t\rightarrow\partial G_t$ such that for $x\in U_t$, $\pi_t(x)$ is a unique point from $\partial G_t$ that realizes dist$(x,\partial G_t)$. By analyzing the proof of that proposition, we see that $U_t$ and $\pi_t$ may be chosen to be $\mc{C}^{k-1}$-continuously dependent on $t$. Moreover, by our assumptions, the choice may be carried out in such a way that for $s,t$ close enough we have $\partial G_t\s U_s$ and $\partial G_s\s U_t.$\\
For each $t$ let us choose open sets
$$
\partial G_t\s U''_t\s\s U'_t\s\s U_t,
$$
varying in a $\mc{C}^{k-1}$-continuous manner with $t$, and cutoff functions $\chi_t$ such that $\chi_t=0$ on $\mb{C}^n\setminus U'_t$ and $\chi_t=1$ on $U''_t$, also varying in a $\mc{C}^{k-1}$-continuous way.\\
If now $s,t$ are close enough, the mapping
$$
\varphi_{st}:\mb{C}^n\ni z \rightarrow z+(\pi_s(z)-\pi_t(z))\chi_s(z)\chi_t(z)\in\mb{C}^n
$$
is a $\mc{C}^{k-1}$-diffeomorphism (cf. [\ref{Stout}], p. 400). Obviously $\varphi_{st}(\partial{G_t})\s\partial G_s$, and even an equality must hold there, because boundary of strongly linearly convex domain of class at least $\mc{C}^2$ is diffeomorphic with $\mb{S}^{{2n-1}}$ (see [\ref{Torres}]), and $\mb{S}^{{2n-1}}$ is not diffeomorphic with any of its proper subsets. Therefore, for $s,t$ close enough $\varphi_{st}$ constitutes a $\mc{C}^{k-1}$-diffeomorphism between $\partial G_t$ and $\partial G_s$ (and indeed between $\overline{G}_t$ and $\overline{G}_s$).\\
\indent We would like to construct a $\mc{C}^{k-1}$-continuous family of $\mc{C}^{k-1}$-diffeomorphisms $(\psi_t:\mb{C}^n\rightarrow\mb{C}^n)_{t\in\sigma'\overline{\mb{D}}}$, mapping $\overline{G_t}$ diffeomorphically to $\overline{G_0}$. Define
$$
\mc{R}:=\{r\in [0,\sigma']:\text{\ there\ exists\ such\ a\ family\ for\ }t\in r\overline{\mb{D}}\}.
$$
Obviously $\mc{R}\neq\varnothing$. Furthermore, $\mc{R}$ is open in $[0,\sigma']$: let $r_0\in\mc{R}$ and let $(\tilde{\psi}_t)_{t\in r_0\overline{\mb{D}}}$ be a suitable family. In virtue of the observation we just made, that for $s,t$ close enough $\varphi_{st}$ is a diffeomorphism between $\overline{G}_t$ and $\overline{G}_0$, we may, without loss of generality, assume that $r_0\neq 0.$ Similarly, we may assume $r_0\neq\sigma'$. It is apparent that we only need to show that there exists some $r\in(r_0,\sigma')$ and the family $(\psi_t)_{t\in r\overline{\mb{D}}}$ of required diffeomorphisms. Let us consider the covering of the circle $|{\zeta}|=r_0$ by the finite family $U_1,\ldots, U_m$ of closed balls, with the multiplicity of the covering equal 2, such that for any $j\in\{1,\ldots, m\}$, if we denote by $a_j$ the center of the ball $U_j$, then $|a_j|<r_0$ and $a_j\notin U_k$ for every $k\in\{1,\ldots,j-1,j+1,\ldots,m\}$, and moreover, for every $j\in\{1,\ldots, m\}$ and for every $s,t\in U_j$, $\varphi_{st}$ is a diffeomorphism between $\overline{G}_t$ and $\overline{G}_s.$ Put
$$
\psi_t:=\begin{cases}
\tilde{\psi}_t,&t\in r_0\overline{\mb{D}}\\
\tilde{\psi}_{d(t)}\circ\varphi_{d(t)t},&t\in\bigcup_{j=1}^mU_j\setminus r_0\overline{\mb{D}},
\end{cases}
$$
where $d(t)$ denotes a point from $|{\zeta}|=r_0$ closest to $t$. Since we take $t$ outside $r_0\overline{\mb{D}}$, we get $d(t)$ is unique, and the function $d$ is smooth.
Then $(\psi_t)_{t\in\left(\bigcup_{j=1}^mU_j\right)\cup r_0\overline{\mb{D}}}$ is a $\mc{C}^{k-1}$ family of required diffeomorphism, which ends the proof of openness of $\mc{R}$ (after restricting the set of parameters to $r\overline{\mb{D}}$ with suitable $r\in(r_0,\sigma')$).\\
\indent $\mc{R}$ is also closed: let $r_{\nu}\rightarrow r_0.$ Then we cover the circle $|{\zeta}|=r_0$ with balls $U_1,\ldots, U_m$ as in the proof of the openness, and we observe that there exists a $\nu_0$ with $\{|{\zeta}|=r_{\nu_0}\}\s\bigcup_{j=1}^mU_j$ and $|a_j|<r_{\nu_0},j=\{1,\ldots,m\}$. Now we use a similar argument as in the proof of openness to produce a suitable family of difeomorphisms for $t\in r_0{\overline{\mb{D}}}\s\left(\bigcup_{j=1}^mU_j\right)\cup r_{\nu_0}\overline{\mb{D}}$.\\
Let $\Phi:\overline{G_0}\rightarrow\overline{I}$ be a $\mc{C}^1$-diffeomorphism, where $I$ is a complete circular domain with boundary of class $\mc{C}^2$ (see [\ref{Pang}]) and let $\Psi:\overline{I}\rightarrow\overline{\mb{B}}$ be a $\mc{C}^1$-diffeomorphism, whose existence was pointed out to the author by Andrea Spiro (private communication; the proof requires some modifications of standard proofs that open star shaped domain is diffeomorphic to the unit ball, cf. [\ref{Elipses}]). The composition of these latter diffeomorphisms gives a $\mc{C}^1$-diffeomorphism between $\overline{G_0}$ and $\overline{\mb{B}}$, and by the argument as in the proof of Theorem 5.2 in [\ref{PAMS18}], there exists $\Theta$, a $\mc{C}^1$-diffeomorphism of $\mb{C}^n$ such that $\Theta(\overline{G}_0)=\overline{\mb{B}}$.  Define $\Gamma_t:=\Theta\circ\psi_t$, which is a $\mc{C}^{k-1}$-continuous family of $\mc{C}^1$-diffeomorphisms of $\mb{C}^n$ with $\Gamma_t(\overline{G}_t)=\overline{\mb{B}}.$ Also, if $R>0$ is large enough, all the mappings $\psi_t$ are equal Id on the preimage by $\Theta$ of the sphere $\|z\|=R.$ Now the family of embedded curves $(\gamma_{t,\zeta})$ as in Theorem \ref{main} may be produced as in the proof of Theorem 5.1 from [\ref{PAMS18}] and with the aid of results from Chapter {3} in [\ref{Hirsch}] and the discussion from [\ref{AMR}, Section 4.1] concerning the parameter dependence of the evolution operators of parameter dependent vector fields.\end{ex}
\begin{rem}
Observe that the argument similar to the one just presented would remain true if only we knew at the very beginning that the closures of the domains we work with are all $\mc{C}^1$-diffeomorphic to the closed unit ball (compare Theorem 5.2 in [\ref{PAMS18}]).
\end{rem}

\section{Concluding remarks}\label{concluding}
As we have observed in the Introduction, it seems that with the existing methods at hand we are not able to omit the assumption about the existence of the family $(\gamma_{t,\zeta})$ in Theorem \ref{main}. On the other hand, even with this additional assumption, the full solution for the question (B) from the Problem \ref{Problem}, i.e. passing from "compact" case to the case where the set of parameters equals $\mb{D}$, and increasing the regularity of the family $(h_{t,\zeta})$ with respect to all variables (if the domains vary in suitably more regular manner) still requires the developing some subtle tools, for example a qualitatively new version of parameter Forstneri\v{c} splitting lemma, where we would have better regularity of the families $(\alpha_t)$ and $(\beta_t)$, and we would be able to change the size of $\mu$- and $\tau$- hulls appearing there pretty arbitrarily with the parameter (see Theorem \ref{Splitting} for the notation). We hope to undertake this problem in forthcoming paper(s).
\bibliographystyle{amsplain}

\end{document}